\documentclass{amsart}

\newtheorem{theorem}{Theorem}[section]

\theoremstyle{definition}
\newtheorem{definition}[theorem]{Definition}

\newtheorem{proposition}[theorem]{Proposition}

\newtheorem{note}[theorem]{Note}
\theoremstyle{remark}

\numberwithin{equation}{section}

\begin{document}

\title{ Metallic structures on the tangent bundle of a  P-Sasakian manifold }

\author{Shahroud Azami}
\address{Department of Mathematics, Faculty of Sciences, Imam Khomeini International University, Qazvin, Iran. }

\email{azami@sci.ikiu.ac.ir}






\begin{abstract}
In this article, we introduce  some metallic structures on the tangent bundle  of a P-Sasakian manifold by complete lift, horizontal lift  and vertical lift of  a P-Sasakian structure $(\phi, \eta,\xi)$ on tangent bundle.  Then we investigate the integrability and parallelity  of these metallic structures.
\end{abstract}

\maketitle
\section{Introduction}
The lift of geometrical objects, vector fields, forms, has an important role in differential geometry. By the  method of lift, we can generalize
 to  differentiable structure on any manifold to its tangent bundle and any other bundles on manifold \cite{MLP, VOP, YI}.
In this paper we study  the metallic structures on  tangent bundle  of a P-Sasakian Riemannian manifold. The metallic  structure is a generalization of the  almost product  structure. A metallic structure  is a polynomial structure as defined by  Goldenberg et al. in \cite{G1, G2}. In  \cite{CEH}, introduced  the notation of metallic structure on a Riemannian manifold.  Suppose that $p$ and $q$ are two positive integers.   The positive solution of the equation  $x^{2}-px-q=0$ is called members  of the metallic means family. These number showed by  $\sigma_{p,q}=\frac{p+\sqrt{p^{2}+4q}}{2}$, where  it is  a generalization of golden proportions.
\begin{definition}
Let $M$ be a manifold.  A metallic  structure on $M$  is an $(1,1)$ tensor field $J$ which satisfies the equation $J^{2}=pJ+qI$, where $p,q$
are positive integers and $I$ is the identity operators on the Lie algebra $\mathcal{X}(M)$ of vector fields on $M$.
If $g$ is a Riemannian metric on $M$, then we say that $g$ is $J$- compatible whenever
\begin{equation}
g(JX,Y)=g(X,JY),\,\,\,\,\,\, \forall X,Y\in\mathcal{X}(M)
\end{equation}
or equivalently
\begin{equation}
g(JX,JY)=pg(X,JY)+qg(X,Y),\,\,\,\,\,\, \forall X,Y\in\mathcal{X}(M).
\end{equation}
In this case $(M,J,g)$ is named a metallic Riemannian manifold.
\end{definition}
Let $J$ be a metallic structure on $M$. Then the Nijenhuise tensor $N_{J}$ of $J$ is a tensor field of type $(1,2)$ given by
$$N_{J}(X,Y)=[JX,JY]-J[JX,Y]-J[X,JY]+J^{2}[X,Y]$$
for $X,Y\in \mathcal{X}(M)$.

On the other hand, at first  time,  Sato in \cite{IS} introduced the P-Sasakian  structure on  a manifolds and studied several properties of these manifolds.
An $n$-dimensional smooth manifold $M$ is called an almost paracontact manifold if it admits an almost paracontact structure $(\phi,\eta,\xi)$, consisting of a $(1,1)$ tensor field $\phi$, a $1$-form $\eta$ and a vector field  $\xi$ which  satisfy the condition
\begin{equation}
\phi^{2}=I-\eta\otimes\xi,\,\,\,\, \eta(\xi)=1,\,\,\,\,\, \phi\xi=0,\,\,\,\,\, \eta\circ \phi=0.
\end{equation}
Let  $g$ be a Riemannian metric compatible with  $(\phi,\eta,\xi)$ i.e.
\begin{equation}
g(X,Y)=g(\phi X,\phi Y)+\eta(X) \eta(Y),\,\,\,\,\, \forall X, Y \in \mathcal{X}(M),
\end{equation}
or equivalently
\begin{equation}
g(X,\phi Y)=g(\phi X,Y),\,\,\,\,\,\,g(X,\xi)=\eta(X) ,\,\,\,\,\, \forall X, Y \in \mathcal{X}(M),
\end{equation}
where $\mathcal{X}(M)$ is the collection of all smooth vector field on $M$. Then $M$ is said to be  an almost paracontact Riemannian manifold.\\
An almost paracontact Riemannian manifold $(M,g)$ is called  a P-Sasakian manifold if it satisfies
\begin{equation}\label{4}
(\nabla _{X}\phi)(Y)=-g(X,Y)\xi-\eta(Y)X+2\eta(X)\eta(Y)\xi
\end{equation}
where $\nabla$ is the Levi-Civita connection of the Riemannain manifold. We have
\begin{equation}
\nabla _{X}\xi=\phi X,\,\,\,\,\,\,\, (\nabla _{X}\eta)(Y)=g(\phi X,Y)= (\nabla _{Y}\eta)(X), \,\,\,\,\,\,\,\forall X, Y \in \mathcal{X}(M).
\end{equation}
\section{Lifts of geometric structure on tangent bundle}
Let $(M,g)$ be a smooth $n$-dimensional Riemannian manifold and   denote its tangent bundle by $TM$. We denote by $\pi : TM\to M$ the natural projection, where it  defines  the natura bundle  structure of $TM$ over $M$ and denote by $T_{l}^{k}(M)$ the set of all tensor fields of the type $(k,l)$ in $M$.  For any point $(x,y)\in TM$, Let $V_{y}=\ker\{\pi_{*}(y):T_{y}(TM)\to T_{x}M\}$ and $VTM=\cup_{y\in TM}V_{y}$. Also, suppose  that $HTM$ be a complement of $VTM$ in $TM$, that is
\begin{equation*}
TTM=VTM\oplus HTM.
\end{equation*}
 $VTM$ and $HTM$ are called  vertical distribution and  horizontal distribution, respectively.
 Suppose that  the  space  $M$ is  covered by a system  of coordinate  neighbourhoods
 $(U,\varphi)=(U,x^{1},x^{2},..., x^{n})$, then the corresponding induced local chart on $TM$ is  $(\pi^{-1}(U), x^{1},x^{2},..., x^{n},y^{1},y^{2},..., y^{n})$.
If in any point of $x\in M$, $\Gamma_{ki}^{h}(x)$ be the Christoffel symbols of $g$, then
the sets of vector fields $\{ \frac{\partial}{\partial y^{1}},...,\frac{\partial}{\partial y^{n}}\}$ and   $\{ \frac{\delta}{\delta x^{1}},...,\frac{\delta}{\delta x^{n}}\}$  on $\pi^{-1}(U)$ define  local frame fields for $VTM$ and $HTM$, respectively, where
$\frac{\delta}{\delta x^{i}}=\frac{\partial}{\partial x^{i}}-y^{k}\Gamma_{ki}^{l}\frac{\partial}{\partial y^{l}}$.
Notice that the set  $\{ \frac{\partial}{\partial y^{i}}, \frac{\delta}{\delta x^{j}} \}$ defines  a local frame on $TM$.
In following from \cite{MLP, YI}, we recall some lifts of geometrical objects  of a manifold to its tangent bundle.
\subsection{Vertical lifts}
Let $f$  be a function  on $M$. Then the vertical  lift of $f$ to $TM$ is the function  $f^{v}$ on $TM$ given by  $f^{v}=f\circ \pi$.
For any vector field $X\in \mathcal{X}(M)$, we define  a vector field  $X^{v}$ in $TM$ by $X^{v}(\iota\omega)=(\omega(X))^{v}$, where $\omega$  is an arbitrary $1$-form in $M$,  We call $X^{v}$ the vertical lift  of $X$. Notice, $X^{v}\in VTM$ and for all function on $M$ we define $X^{v}(df)=X.f$.
Let $F$ be  a tensor field of type $(1,r)$ or $(0,r)$, $r\geq 1$ on $M$. Then  the vertical lift of  $F$ on $TM$ defined by
\begin{equation*}
F_{y}^{v}(\tilde{X}_{1}, \tilde{X}_{2}, ..., \tilde{X}_{r})=\Big( F_{y}\big(\pi_{\star}(\tilde{X}_{1}), \pi_{\star}(\tilde{X}_{2}), ..., \pi_{\star}(\tilde{X}_{r})\big)\Big)^{v}
\end{equation*}
where $\tilde{X}_{1}, \tilde{X}_{2}, ..., \tilde{X}_{r}\in T_{y}(TM),\,\,\, y\in T_{x}M,\,\,\,x\in M$. Hence  for any $X_{1},..., X_{r}\in \mathcal{X}(M)$ we have
\begin{equation*}
F^{v}({X}_{1}^{v}, {X}_{2}^{v}, ..., {X}_{r}^{v})=0,\,\,F^{v}({X}_{1}^{c}, {X}_{2}^{c}, ..., {X}_{r}^{c})=\Big( F({X}_{1}, {X}_{2}, ...,{X}_{r}\big)\Big)^{v}.
\end{equation*}
\subsection{Complete lifts}
 If $f$ is a function in $M$, then  the complete  lift of  $f$  is the function  $f^{c}$ on $TM$  and defined by
\begin{equation}
f^{c}(x,y)=df(x)(y),\,\,\, y\in T_{x}M,\,\, x\in M.
\end{equation}
Also the complete lift of vector field  $X=X^{i}\frac{\partial}{\partial x^{i}}$ on $M$ defined by
\begin{equation*}
X^{c}=X^{i}\frac{\partial}{\partial x^{i}}-y^{j}\frac{\partial X^{i}}{\partial x^{j}}\frac{\partial}{\partial y^{i}}.
\end{equation*}
Therefore we obtain $(\frac{\partial}{\partial x^{i}})^{c}=\frac{\partial}{\partial x^{i}}$ and $X^{c}f^{c}=(Xf)^{c}$ for any function $f$ on $M$.
Suppose that $\omega$  is a $1$-form on $M$. The complete lift of  $\omega$ on $TM$ defined by $\omega^{c}(X^{c})=(\omega(X))^{c},\,\,\,\,\omega^{c}(X^{v})=(\omega(X))^{v},\,\,\,X\in\mathcal{X}(M)$. In general case, the complete lift of  a tensor field $F$ of  type $(1,r)$ or $(0,r)$, $r\geq 1$, on $M$ defined  by
 $F^{c}(X_{1}^{c},...,X_{r}^{c})=(F(X_{1},..., X_{r}))^{c}$ for any $X_{1},..., X_{r}\in \mathcal{X}(M)$. Then the  complete lift of a Riemannian metric $g$ defined by
$$g^{c}=\left(  \begin{array}{cc}    y^{k}\frac{\partial g_{ij}}{\partial x^{k}} & g_{ij}\\    g_{ij} & 0 \\  \end{array}\right).$$
From \cite{MLP} and \cite{YI}, we have
\begin{proposition}\label{p1}
Let $M$ be a manifold with a Riemannain metric $g$. For any $X,Y\in \mathcal{X}(M),\, f\in C^{\infty}(M)$, and  $(1,1)$ tensor field $F$ we have :
\begin{itemize}
\item $X^{v}f^{v}=0,\,\, X^{v}f^{c}=X^{c}f^{v}=(Xf)^{v},\,\, X^{c}f^{c}=(Xf)^{c}$,
\item $F^{c}(X^{v})=(F(X))^{v}$,
\item $g^{c}$ is semi-Riemannian metric and
$$g^{c}(X^{v},Y^{c})=g^{c}(X^{c},Y^{v})=(g(X,Y))^{v},\,\,g^{c}(X^{v},Y^{v})=0,\,\,g^{c}(X^{c},Y^{c})=(g(X,Y))^{c}$$,
\item If $P(x)$ is a polynomial in one variable $x$ then $P(F^{c})=(P(F))^{c}$.
\end{itemize}
\end{proposition}
We define the complete lift of a linear connection  $\nabla$ to $TM$ as the unique linear connection $\nabla^{c} $ on $TM$ as
 $\nabla^{c}_{X^{c}}Y^{c}=(\nabla_{X}Y)^{c}$ for $X,Y\in \mathcal{X}(M)$, therefore
\begin{eqnarray*}
\nabla^{c}_{\frac{\partial}{\partial x^{i}}}\frac{\partial}{\partial x^{j}}&=&\Gamma_{ij}^{k}\frac{\partial}{\partial x^{k}}+
y^{l}\frac{\partial \Gamma_{ij}^{k}}{\partial x^{l}}\frac{\partial}{\partial y^{k}},\,\,
\nabla^{c}_{\frac{\partial}{\partial y^{i}}}\frac{\partial}{\partial y^{j}}=0,\\
\nabla^{c}_{\frac{\partial}{\partial x^{i}}}\frac{\partial}{\partial y^{j}}&=&\Gamma_{ij}^{k}\frac{\partial}{\partial y^{k}},\,\,
\nabla^{c}_{\frac{\partial}{\partial y^{i}}}\frac{\partial}{\partial x^{j}}=\Gamma_{ij}^{k}\frac{\partial}{\partial y^{k}}.
\end{eqnarray*}
\begin{proposition}[\cite{MLP, YI}] Let  $T$ and $R$ be  the torsion and curvature tensor of $\nabla$, respectively. Then $T^{c}$ and $R^{c}$ are  the   torsion and curvature tensor of $\nabla^{c}$, respectively, and
\begin{itemize}
\item $\nabla $ is symmetric if and only if $\nabla^{c}$ is symmetric.
\item $\nabla $ is flat if and only if $\nabla^{c}$ is flat.
\end{itemize}
\end{proposition}
\begin{proposition}[\cite{MLP, YI}]
Let $F$ be a tensor field of type $(1,r)$ or $(0,r)$, $r\geq 1$,  on $M$ and $X,Y\in \mathcal{X}(M)$, we have
\begin{eqnarray*}
\nabla^{c}_{X^{v}}Y^{v}&=&(\nabla_{X}Y)^{v},\,\, \nabla^{c}_{X^{c}}Y^{c}=(\nabla_{X}Y)^{c},\,\,\nabla^{c}_{X^{v}}Y^{c}=\nabla^{c}_{X^{c}}Y^{v}=(\nabla_{X}Y)^{v},\\
\nabla^{c}F^{v}&=&(\nabla F)^{v},\,\,\nabla^{c} F^{c}=(\nabla F)^{c}.
\end{eqnarray*}
\end{proposition}

\subsection{Horizontal lifts}
 The horizontal  lift $f^{h}$ of a function $f$ on $M$ is  given by $f^{h}=f^{c}-\nabla_{\gamma}f$ where $\nabla_{\gamma}f=\gamma(\nabla f)$, where for any tensor field $F$ of type $(1,r)$ or $(0,r)$, $r\geq 1$,  on $M$, $\gamma_{X}F=(F_{X})^{v}$ and $F_{X}(X_{1},..., F_{r-1})=F(X_{1},..., X_{r-1}, X)$.
For any vector field  $X=X^{i}\frac{\partial}{\partial x^{i}}$ on $M$ there exists a unique vector $X^{h}\in HVM$ such that $\pi_{*}X^{h}=X$,
that is if $X=X^{i}\frac{\partial}{\partial x^{i}}$ then $X^{h}=X^{i}\frac{\partial}{\partial x^{i}}-y^{j}X^{i}\Gamma_{ij}^{l}\frac{\partial}{\partial y^{l}}$. We call $X^{h}$ the horizontal lift of $X$ in  the  point $(x,y)\in TM$.
Let $\omega$ be a $1$-form on $M$. Then the horizontal lift $\omega^{h}$ of $\omega$ is defined by $\omega^{h}=\omega^{c}-\nabla_{\gamma}\omega$. Then for any $X\in \mathcal{X}(M)$ we have  $\omega^{h}(X^{h})=0,\,\,\omega^{h}(X^{v})=(\omega(X))^{v}$.
 The horizontal  lift of a $(1,1)$ tensor field  $F$ on $M$ defined by $F^{h}(X^{h})=(FX)^{h},\,\, F^{h}({X^{v}})=(FX)^{v}$.
The horizontal  lifts of a Riemannian metric $g$ defined by
$g^{h}=g_{ij}\theta^{i}\otimes \eta^{j}+g_{ij} \eta^{i}\otimes\theta^{j}$ where $\theta^{i}=dx^{i},\,\, \eta^{i}=y^{j}\Gamma_{jk}^{i}dx^{k}+dy^{i}$.\\
From \cite{MLP, YI}, we have
\begin{proposition}\label{p}
Let $M$ be a manifold with a Riemannain metric $g$.  For any $X,Y\in \mathcal{X}(M),\, f\in C^{\infty}(M)$, and  $(1,1)$ tensor field $F$ we have :
\begin{itemize}
\item $g^{h}$ is semi-Riemannian metric and $g^{h}(X^{v},Y^{h})=(g(X,Y))^{v},$ $g^{h}(X^{v},Y^{v})=g^{h}(X^{h},Y^{h})=0$,
\item If $P(x)$ is a polynomial in one variable $x$ then $P(J^{h})=(P(J))^{h}$,
\item $g^{h}=g^{c}$ if and only if $\nabla g=0$,
\item $ [X^{v},Y^{v}]=0,\,\,[X^{v},Y^{c}]=[X,Y]^{v},\,\, [X^{c},Y^{c}]=[X,Y]^{c}$,\\
$[X^{v},Y^{h}]=-(\nabla_{Y}X)^{v},\,\, [X^{h},Y^{h}]=[X,Y]^{h}-\gamma R(X,Y)$ where $R$ is curvature tensor of $g$ and the vertical vector lift $\gamma F$ defined by $(\gamma F)(y)=(F(y))^{v}$.
\end{itemize}
\end{proposition}
 Let $\nabla$ be a linear connection on $M$, then we define the horizontal lift of $\nabla$ to $TM$ as the unique linear connection $\nabla^{h} $ on $TM$ given by
$$
\nabla^{h}_{X^{v}}Y^{v}=\nabla^{h}_{X^{v}}Y^{h}=0,\,\,\nabla^{h}_{X^{h}}Y^{v}=(\nabla_{X}Y)^{v},\,\,
\nabla^{h}_{X^{h}}Y^{h}=(\nabla_{X}Y)^{h},
$$
for $X,Y\in \mathcal{X}(M)$. Hence $\nabla^{h}_{X^{c}}Y^{c}=(\nabla_{X}Y)^{c}-\gamma R(.\,,X,Y)$ where $R(.\,,X,Y)Z=R(Z,X,Y)$.

\section{Metallic structures on the tangent bundle of a  P-Sasakian manifold}
Let $M$ be an $n$-dimensional P-Sasakian manifold with structure tensor $(\phi, \eta,\xi,g)$. In this part of this section we study the metallic structure induced on $TM$  by the complete lift of  a  P-Sasakian  structure.
\begin{proposition}
On the tangent bundle of  P-Sasakian manifold with structure tensor $(\phi, \eta,\xi,g)$, there exists  a metallic structure given by
\begin{equation}\label{e1}
J=\frac{p}{2}I-(\frac{2\sigma_{p,q}-p}{2})(\phi^{c}+\eta^{v}\otimes\xi^{v}+\eta^{c}\otimes\xi^{c}).
\end{equation}
\end{proposition}
\begin{proof}
From the definition of the almost paracontact structure of  P-Sasakian manifold, we obtain the following relations
\begin{eqnarray*}
&&(\phi^{c})^{2}=(\phi^{2})^{c}=I-\eta^{c}\otimes \xi^{v}-\eta^{v}\otimes \xi^{c},\\
&&\eta^{v}( \xi^{c})=\eta^{c}( \xi^{v})=1,\,\,\,\,\,\,\eta^{v}( \xi^{v})=\eta^{c}( \xi^{c})=0,\\
&&\phi^{c}( \xi^{v})=\phi^{c}( \xi^{c})=0,\,\,\,\,\,\,\eta^{v}\circ \phi^{c}=\eta^{c}\circ \phi^{c}=0.
\end{eqnarray*}
Therefore for any $\tilde{X}\in \mathcal{X}(TM)$ we have
\begin{eqnarray*}
&&J(\xi^{v})=\frac{p}{2}\xi^{v}-\frac{2\sigma_{p,q}-p}{2}\xi^{c},\,\,\,\,\,\,J(\xi^{c})=\frac{p}{2}\xi^{c}-\frac{2\sigma_{p,q}-p}{2}\xi^{v},\\
&&J(\phi^{c}\tilde{X})=\frac{p}{2}\phi^{c}\tilde{X}-\frac{2\sigma_{p,q}-p}{2}(\tilde{X}-\eta^{c}(\tilde{X})\xi^{v}-\eta^{c}(\tilde{X})\xi^{c}).
\end{eqnarray*}
Now, we obtain
\begin{equation*}
J(\tilde{X})=\frac{p}{2}\tilde{X}-\frac{2\sigma_{p,q}-p}{2}(\phi^{c}\tilde{X}+\eta^{v}(\tilde{X})\xi^{v}+\eta^{c}(\tilde{X})\xi^{c})
\end{equation*}
and
\begin{equation*}
J^{2}(\tilde{X})=\frac{p}{2}J(\tilde{X})-\frac{2\sigma_{p,q}-p}{2}(J(\phi^{c}\tilde{X})+\eta^{v}(\tilde{X})J(\xi^{v})+\eta^{c}(\tilde{X})J(\xi^{c}))=pJ(\tilde{X})+q \tilde{X},
\end{equation*}
and it complete the proof.
\end{proof}
\begin{proposition}
If $M$ is  a   P-Sasakian manifold with structure tensor $(\phi, \eta,\xi,g)$ and  $J$ is defined by (\ref{e1})  then we have
\begin{equation}
g^{c}(J\tilde{X}, J\tilde{Y})=pg^{c}(\tilde{X}, J\tilde{Y})+qg^{c}(\tilde{X}, \tilde{Y}),\,\,\,\,\,\forall \tilde{X}, \tilde{Y} \in \mathcal{X}(TM).
\end{equation}
\end{proposition}
\begin{proof}
For any $X,Y\in \mathcal{X}(M)$, we have
\begin{eqnarray*}
&&g^{c}(X^{v},Y^{v})=0,\,\,\,g^{c}(X^{v},\xi^{c})=(g(X,\xi))^{v}=(\eta(X))^{v},\\
&&g^{c}((\phi X)^{v},\xi^{c})=(g(\phi X,\xi))^{v}=(g(X,\phi \xi))^{v}=0,\\
&&g^{c}(\xi^{c},\xi^{c})=(g(\xi,\xi))^{c}=0.
\end{eqnarray*}
Therefore
\begin{equation*}
g^{c}(JX^{v},JY^{v})=-\frac{(2\sigma_{p,q}-p)p}{2}(\eta X)^{v}(\eta Y)^{v},
\end{equation*}
and
\begin{equation*}
g^{c}(X^{v},JY^{v})=-\frac{2\sigma_{p,q}-p}{2}(\eta X)^{v}(\eta Y)^{v}.
\end{equation*}
Thus
\begin{equation*}
g^{c}(JX^{v},JY^{v})=pg^{c}(X^{v},JY^{v})+qg^{c}(X^{v},Y^{v}).
\end{equation*}
Also, using $g^{c}(X^{v},Y^{c})=(g(X,Y))^{v}$ and $g^{c}(X^{c},Y^{c})=(g(X,Y))^{c}$ we have
\begin{equation*}
g^{c}(JX^{v},JY^{c})=(\frac{p^{2}}{2}+q)(g(X,Y))^{v} -\frac{(2\sigma_{p,q}-p)p}{2}[(g(X,\phi Y))^{v}-(\eta X)^{v}(\eta Y)^{v}],
\end{equation*}
and
\begin{equation*}
g^{c}(X^{v},JY^{c})=\frac{p}{2}(g(X,Y))^{v} -\frac{(2\sigma_{p,q}-p)p}{2}[(g(X,\phi Y))^{v}-(\eta X)^{v}(\eta Y)^{v}].
\end{equation*}
Hence
\begin{equation*}
g^{c}(JX^{v},JY^{c})=pg^{c}(X^{v},JY^{c})+qg^{c}(X^{v},Y^{c}).
\end{equation*}
The other cases are similar.
\end{proof}

\begin{theorem}
Let  $M$ be   a   P-Sasakian manifold with structure tensor $(\phi, \eta,\xi,g)$ and  $J$ be defined by (\ref{e1})  then the metallic structure $J$ is integrable.
\end{theorem}
\begin{proof}
The $1$-form $\eta$ defines on $(n-1)$-dimensional distribution $\mathcal{D}$  follows
\begin{equation}\label{df1}
\forall p\in M,\,\,\,\,\,\,\,\mathcal{D}_{p}=\{v\in T_{p}M: \eta(v)=0\},
\end{equation}
the complement of $\mathcal{D}$  is $1$-dimensional distribution spanned by $\xi$. Suppose that
\begin{equation*}
N^{1}=N_{\phi}-2d\eta\otimes \xi ,\,\, N^{2}(X,Y)=(L_{\phi X}\eta)Y-(L_{\phi Y}\eta)X,\,\,N^{3}=L_{\xi}\phi,\,\,N^{4}=L_{\xi}\eta.
\end{equation*}
By Proposition \ref{p} for any $X,Y$ of $C^{\infty}(M)$-module of all sections of distribution $\mathcal{D}$, we have
\begin{eqnarray*}
N_{J}(X^{v},Y^{v})&=&0,\\
N_{J}(X^{v},Y^{c})&=&A\Big([N^{1}(X,Y)]^{v}+N^{2}(X,Y)\xi^{c} \Big),\\
N_{J}(X^{c},Y^{c})&=&A\Big([N^{1}(X,Y)]^{c}+N^{2}(X,Y)\xi^{v} \Big),\\
N_{J}(X^{v},\xi^{v})&=&A\Big( -(N^{3}(X))^{v}+N^{4}(X)\xi^{c}\Big),\\
N_{J}(X^{v},\xi^{c})&=&A\Big([\phi(N^{3}(X))-N^{4}(X)\xi]^{v}+N^{2}(X,\xi)\xi^{c} \Big),\\
N_{J}(X^{c},\xi^{v})&=&A\Big(-(N^{3}(X))^{c}+(\phi(N^{3}(X)))^{v}-[N^{4}(\phi X)-N^{4}(X)]^{c}\xi^{c}\Big),\\
N_{J}(X^{c},\xi^{c})&=&A\Big( -(N^{3}(X))^{v}+N^{4}(X)+N^{2}(X,\xi)]\xi^{c}+\phi(N^{3}(X))-N^{4}(X)\xi]^{c}\Big),\\
N_{J}(\xi^{v},\xi^{v})&=&N_{J}(\xi^{c},\xi^{c})=N_{J}(\xi^{v},\xi^{c})=0,
\end{eqnarray*}
where $A=(\frac{2\sigma_{p,q}-p}{2})^{2}$. But the tensor $N^{1}$ of a P-Sasakian manifold vanishes. On the other hand if $N^{1}=0$ then also $N^{2}, N^{3}$ and $N^{4}$ vanish. Hence  $N_{J}(\tilde{X},\tilde{Y})=0$ for all $\tilde{X},\tilde{Y}\in \mathcal{X}(TM)$, that is $J$ is integrable.
\end{proof}
\begin{theorem}
Let  $M$ be   a   P-Sasakian manifold with structure tensor $(\phi, \eta,\xi,g)$ and  $J$ be defined by (\ref{e1})  then the metallic structure $J$ is never parallel with respect      to $\nabla^{c}$.
\end{theorem}
\begin{proof}
We have
\begin{eqnarray}
(\nabla^{c}_{X^{c}}J)\xi^{c}&=&\nabla^{c}_{X^{c}}(J\xi^{c})-J(\nabla^{c}_{X^{c}}\xi^{c})\\\nonumber
&=&-\frac{2\sigma_{p,q}-p}{2}\left [\nabla^{c}_{X^{c}}\big((\phi\xi)^{c}+\eta(\xi)^{v}\xi^{v}+\eta(\xi)^{c}\xi^{c}\big)-(\phi\nabla_{X} \xi)^{c}\right.\\\nonumber
&&\left. -(\eta(\nabla_{X}\xi))^{v}\xi^{v}-(\eta(\nabla_{X}\xi))^{c}\xi^{c}\right].
\end{eqnarray}
Using $\phi X=\nabla_{X}\xi$ we get
\begin{equation}
(\nabla^{c}_{X^{c}}J)\xi^{c}=-\frac{2\sigma_{p,q}-p}{2}[(\phi X)^{v}-X^{c}]\neq0,\,\,\,\,\forall X\in\mathcal{D}\setminus\{0\},
\end{equation}
where $\mathcal{D}$ is   a distribution and  defined by (\ref{df1}).
\end{proof}
\begin{proposition}\label{pp1}
Let $M$ be a P-Sasakian manifold with structure tensor $(\phi, \eta,\xi,g)$, $\nabla\phi=0$  and   $J$ be defined by (\ref{e1}). Then  fundamental $2$-form $\Phi$, given by
\begin{equation}
\Phi(\tilde{X},\tilde{Y})=g^{c}(\tilde{X},J\tilde{Y})-\frac{p}{2}g^{c}(\tilde{X},\tilde{Y}),\,\,\,\,\,\,\tilde{X},\tilde{Y}\in \mathcal{X}(TM),
\end{equation}
is closed if and only if
\begin{equation}\label{27}
g(\nabla_{Y}X, \phi Z)+g(\nabla_{Z}Y, \phi X)+g(\nabla_{X}Z, \phi Y)=0,\,\,\,\,\,\,\forall X,Y,Z\in \mathcal{X}(M).
\end{equation}
\end{proposition}
\begin{proof}
The cobundary formula for $d$ on a $2$-form $\Phi$ is
\begin{eqnarray*}
3d\Phi(\tilde{X},\tilde{Y},\tilde{Z})&=&\tilde{X}\Phi(\tilde{Y},\tilde{Z})+\tilde{Y}\Phi(\tilde{Z},\tilde{X})+\tilde{Z}\Phi(\tilde{X},\tilde{Y})\\
&&-\Phi([\tilde{X},\tilde{Y}],\tilde{Z})-\Phi([\tilde{Z},\tilde{X}],\tilde{Y})-\Phi([\tilde{Y},\tilde{Z}],\tilde{X})
\end{eqnarray*}
for any $\tilde{X},\tilde{Y},\tilde{Z}\in \mathcal{X}(TM) $. Hence for any $X,Y,Z\in \mathcal{X}(M)$ we have
\begin{eqnarray*}
3d\Phi(X^{c},Y^{c},Z^{v})&=&X^{c}g^{c}(Y^{c},JZ^{v})-\frac{p}{2}X^{c}g^{c}(Y^{c},Z^{v})+Y^{c}g^{c}(Z^{v},JX^{c})\\&&-\frac{p}{2}Y^{c}g^{c}(Z^{v},X^{c})+Z^{v}g^{c}(X^{c},JY^{c})-\frac{p}{2}Z^{V}g^{c}(X^{c},Y^{c})\\&&
-g^{c}([X,Y]^{c},JZ^{v})+\frac{p}{2}g^{c}([X,Y]^{c},Z^{v})-g^{c}([Z,X]^{v},JY^{c})\\&&
+\frac{p}{2}g^{c}([Z,X]^{v},Y^{c})
-g^{c}([Y,Z]^{v},JX^{c})+\frac{p}{2}g^{c}([Y,Z]^{v},X^{c}).
\end{eqnarray*}
On the other hand
\begin{equation*}
JZ^{v}=\frac{p}{2}Z^{v}-\frac{2\sigma_{p,q}-p}{2}\big((\phi(Z))^{v}+(\eta(Z))^{v}\xi^{c}\big),
\end{equation*}
and
\begin{equation*}
JX^{c}=\frac{p}{2}X^{c}-\frac{2\sigma_{p,q}-p}{2}\big((\phi(X))^{v}+(\eta(X))^{v}\xi^{v}+(\eta(X))^{c}\xi^{c}\big).
\end{equation*}
Therefore,
\begin{eqnarray*}\nonumber
-\frac{6}{2\sigma_{p,q}-p}d\Phi(X^{c},Y^{c},Z^{v})&=&\left\{ Xg(Y,\phi Z)+Yg(Z,\phi X)+Zg(X,\phi Y)\right.\\&&\left.-g([X,Y],\phi Z)-g([Z,X],\phi Y)-g([Y,Z],\phi X)\right\}^{v}\\&&
+X^{c}[(\eta(Z))^{v}(\eta(Y))^{c}]+Y^{c}[(\eta(X))^{c}(\eta(Z))^{v}]\\&&
+Z^{v}[(\eta(Y))^{v}(\eta(X))^{v}]+Z^{v}[(\eta(Y))^{c}(\eta(X))^{c}]\\&&
-(\eta(Z))^{v}g([X,Y],\xi)^{c}-(\eta(Y))^{c}g([Z,X],\xi)^{v}\\&&
-(\eta(X))^{c}g([Y,Z],\xi)^{v}.
\end{eqnarray*}
Since  $\nabla \phi=0$ and $g(X,\phi Y)=g(Y,\phi X)$, we get
\begin{equation*}
Xg(Y,\phi Z)-g([X,Y],\phi Z)=g(\nabla_{Y}X,\phi Z)+g(\nabla_{X}Z,\phi Y).
\end{equation*}
Also, $\nabla_{X}\xi =\phi X$ results that
\begin{eqnarray*}
g([X,Y],\xi)&=&g(\nabla_{X}Y,\xi)-g(\nabla_{Y}X,\xi)\\
&=&Xg(Y,\xi)-g(Y,\nabla_{X}\xi)-Yg(X,\xi)+g(X,\nabla_{Y}\xi)\\
&=&X(\eta(Y))-g(Y,\phi X)-Y(\eta(X))+g(X,\phi Y)\\
&=&X(\eta(Y))-Y(\eta(X)).
\end{eqnarray*}
Thus
\begin{eqnarray*}
-\frac{6}{2\sigma_{p,q}-p}d\Phi(X^{c},Y^{c},Z^{v})&=&2\left\{ g(\nabla_{Y}X,\phi Z)+g(\nabla_{Z}Y,\phi X)+g(\nabla_{X}Z,\phi Y)\right\}\\
&&+2\left\{(\eta(Y))^{c} (X\eta(Z))^{v}+(\eta(Z))^{v}(Y\eta(X))^{c}\right.\\&&\left.+(\eta(X))^{c}(Z\eta(Y))^{v}\right\}
\end{eqnarray*}
Now, if $X,Y,Z\in \mathcal{D}$ , then $d\Phi(X^{c},Y^{c},Z^{v})=0$ is equivalent  with (\ref{27}), where $\mathcal{D}$ is   a distribution and  defined by (\ref{df1}). Also, if $X=\xi$ or $Z=\xi$ , we get the same result.
The other cases reducible   to  (\ref{27}).
\end{proof}
\begin{note}
$\frac{p}{2}I-(\frac{2\sigma_{p,q}-p}{2})(\phi^{c}\pm\eta^{v}\otimes\xi^{v}\pm\eta^{c}\otimes\xi^{c})$ are also metallic structures. For these structures we can obtain the similar results as for the metallic structure (\ref{e1}).
\end{note}
In the following we study a metallic structure on $TM$ induced by the horizontal lift.
\begin{proposition}
Let $M$ be a P-Sasakian manifold with structure tensor $(\phi, \eta,\xi,g)$. Then
there exists  a metallic structure on its  tangent bundle, given by
\begin{equation}\label{e2}
F=\frac{p}{2}I-(\frac{2\sigma_{p,q}-p}{2})(\phi^{h}+\eta^{h}\otimes\xi^{h}+\eta^{v}\otimes\xi^{v}).
\end{equation}
\end{proposition}
\begin{proof}
By definition vertical lift and horizontal lift of the almost paracontact structure of P-Sasakian manifold, we have
\begin{eqnarray*}
&&(\phi^{h})^{2}=(\phi^{2})^{h}=I-\eta^{h}\otimes \xi^{v}-\eta^{v}\otimes \xi^{h},\\
&&\eta^{v}( \xi^{h})=\eta^{h}( \xi^{v})=1,\,\,\,\,\,\,\eta^{v}( \xi^{v})=\eta^{h}( \xi^{h})=0,\\
&&\phi^{h}( \xi^{v})=\phi^{h}( \xi^{h})=0,\,\,\,\,\,\,\eta^{v}\circ \phi^{h}=\eta^{h}\circ \phi^{h}=0.
\end{eqnarray*}
Therefore for any $\tilde{X}\in \mathcal{X}(TM)$ we have
\begin{eqnarray*}
&&F(\xi^{v})=\frac{p}{2}\xi^{v}-\frac{2\sigma_{p,q}-p}{2}\xi^{h},\,\,\,\,\,\,J(\xi^{h})=\frac{p}{2}\xi^{h}-\frac{2\sigma_{p,q}-p}{2}\xi^{v},\\
&&F(\phi^{h}\tilde{X})=\frac{p}{2}\phi^{h}\tilde{X}-\frac{2\sigma_{p,q}-p}{2}(\tilde{X}-\eta^{h}(\tilde{X})\xi^{v}-\eta^{v}(\tilde{X})\xi^{h}).
\end{eqnarray*}
Now, we obtain
\begin{equation*}
F(\tilde{X})=\frac{p}{2}\tilde{X}-\frac{2\sigma_{p,q}-p}{2}(\phi^{h}\tilde{X}+\eta^{v}(\tilde{X})\xi^{v}+\eta^{h}(\tilde{X})\xi^{h})
\end{equation*}
and
\begin{equation*}
F^{2}(\tilde{X})=\frac{p}{2}F(\tilde{X})-\frac{2\sigma_{p,q}-p}{2}(F(\phi^{h}\tilde{X})+\eta^{v}(\tilde{X})F(\xi^{v})+\eta^{h}(\tilde{X})F(\xi^{h}))=pF(\tilde{X})+q \tilde{X},
\end{equation*}
and it finish the proof.
\end{proof}
\begin{definition}[Sasakian  metric ] Let $(M,g)$ be a Riemannian manifold. The Sasakian metric on $TM$ defined as follows
\begin{equation}\label{}
G(X^{v},Y^{h})=0,\,\, G(X^{v},Y^{v})=[g(X,Y)]^{v},\,\,\,G(X^{h},Y^{h})=[g(X,Y)]^{v},
\end{equation}
for any $X,Y\in\mathcal{X}(M)$.
\end{definition}
\begin{proposition}Let $M$ be a P-Sasakian manifold with structure tensor $(\phi, \eta,\xi,g)$. If $F$ defined by (\ref{e2}) on $TM$ and $G$ be Sasakian metric, then we have
\begin{equation}\label{e3}
G(F\tilde{X},F\tilde{Y})=pG(\tilde{X},F\tilde{Y})+qG(\tilde{X},\tilde{Y}).
\end{equation}
for any $\tilde{X}, \tilde{Y} \in \mathcal{X}(TM)$.
\end{proposition}
\begin{proof}
For any $X,Y\in \mathcal{X}(M)$, we have
\begin{equation*}
\eta^{h}(X^{v})=\eta^{v}(X^{v})=0,\,\,\eta^{h}X^{h}=\eta^{v}X^{h}=(\eta(X))^{h},\,\,\phi^{h}X^{v}=(\phi X)^{v},\,\,\phi^{h}X^{h}=(\phi X)^{h}.
\end{equation*}
Hence
\begin{equation*}
FX^{v}=\frac{p}{2}X^{v}-\frac{2\sigma_{p,q}-p}{2}(\phi X)^{v}.
\end{equation*}
Now by definition of Sasakian metric we get
\begin{equation*}
G(FX^{v},FY^{v})=\left\{(q+\frac{p^{2}}{2})g(X,Y)-\frac{2\sigma_{p,q}-p}{2}pg(X,\phi Y)\right\}^{v}.
\end{equation*}
and
\begin{equation*}
G(X^{v},FY^{v})=\left\{\frac{p}{2}g(X,Y)-\frac{2\sigma_{p,q}-p}{2}g(X,\phi Y)\right\}^{v}.
\end{equation*}
Therefore
\begin{equation*}
G(FX^{v},FY^{v})=pG(X^{v},FY^{v})+qG(X^{v},Y^{v}).
\end{equation*}
the other cases are similar.

\end{proof}
\begin{definition} Let $(M,g)$ be a Riemannian manifold and  $\nabla$  the Levi-Civita
connection with respect to the Riemannian metric $g$. Let $\mathcal{D}$ be    a distribution and  defined by (\ref{df1}). Connection $\nabla$   is called $\mathcal{D}$-flat if $\nabla_{X}Y\in \mathcal{D}$ for all $X,Y\in \mathcal{D}$.
\end{definition}
\begin{theorem}
Let $M$ be  a  P-Sasakian manifold with structure tensor $(\phi, \eta,\xi,g)$. Then the metallic structure $F$ defined by (\ref{e2}) on $TM$  is integrable  if and only if $\nabla$ is $\mathcal{D}$-flat and
\begin{equation}\label{e4}
R(\phi X,\phi Y)+R(X,Y)-\phi\{R(\phi X, Y)+R(X,\phi Y)\}=0,
\end{equation}
where $R$ is curvature tensor  of $M$.
\end{theorem}
\begin{proof}
 Let $X,Y\in \mathcal{D}$ and $U\in TM$. We have
\begin{eqnarray*}
-\frac{2}{2\sigma_{p,q}-p}N_{F}(X^{h},Y^{h})U&=&[N^{1}(X,Y)]^{h}U+\{\eta R(\phi X, Y)U+\eta R(X,\phi Y)U\}\xi^{h}\\
&&-\{R(\phi X,\phi Y)U+R(X,Y)U-\phi R(\phi X, Y)U\\&&+R(X,\phi Y)U\}^{v}+N^{2}(X,Y)\xi^{v}.
\end{eqnarray*}
 Also
\begin{eqnarray*}
-\frac{2}{2\sigma_{p,q}-p}N_{F}(X^{h},Y^{v})&=&\big(\nabla_{\phi X}\phi Y-\phi \nabla_{\phi X}Y-\phi \nabla_{ X}\phi Y+\nabla_{X}Y\big)^{v}\\&&-\{\eta(\nabla_{\phi X}Y)+\eta(\nabla_{ X}\phi Y)\}\xi^{h},\\
N_{F}(X^{v},Y^{v})&=&0,\\
-\frac{2}{2\sigma_{p,q}-p}N_{F}(X^{h},\xi^{h})U&=&\big(\nabla_{\phi X}\xi-\phi(\nabla_{X}\xi)+\phi(N^{3}(X)) \big)^{h}U+[\eta R(\phi X,\xi)U]\xi^{h}\\&&
+\{N^{2}(X,\xi)-\eta(\nabla_{X}\xi)\}\xi^{v}+[\phi R(\phi X,\xi)U]^{v},\\
-\frac{2}{2\sigma_{p,q}-p}N_{F}(X^{h},\xi^{v})U&=&-[N^{3}(X)]^{h}U+\{N^{2}(X,\xi)+\eta R(X,\xi)U\}\xi^{h}\\&&
-\left\{R(\phi X,\xi)U-\phi R(X,\xi)U+\phi(\nabla_{\phi X}\xi)-\nabla_{X}\xi \right\}^{v}\\&&-N^{4}(X)\xi^{v},
\end{eqnarray*}
and
\begin{equation*}
-\frac{2}{2\sigma_{p,q}-p}N_{F}(X^{v},\xi^{v})=-(\nabla_{\xi}\phi X-\phi\nabla_{\xi}X)^{v}+\eta(\nabla_{\xi}X)\xi^{h}.
\end{equation*}
Hence $N_{F}=0$ if and only if (\ref{e4}) is true and
\begin{equation}\label{e5}
\nabla_{\phi X}\phi Y-\phi \nabla_{\phi X}Y-\phi \nabla_{ X}\phi Y+\nabla_{X}Y=0.
\end{equation}
From (\ref{4}), (\ref{e5}) is equivalent with $\eta(\nabla_{X}Y)=0$, for any $X,Y\in \mathcal{D}$ ,i.e. $\nabla$ is $\mathcal{D}$-flat.
\end{proof}
\begin{theorem}
Let $M$ be  a  P-Sasakian manifold with structure tensor $(\phi, \eta,\xi,g)$. Then the metallic structure $F$ defined by (\ref{e2}) on $TM$  is never parallel with respect to $\nabla^{h}$.
\end{theorem}
\begin{proof}
We have
\begin{equation*}
(\nabla_{X^{h}}^{h}F)\xi^{h}=\nabla_{X^{h}}^{h}(F\xi^{h})-F(\nabla_{X^{h}}^{h}\xi^{h})=-\frac{2\sigma_{p,q}-p}{2}[(\phi X)^{v}-(\phi^{2}X)^{h}].
\end{equation*}
If $X\in \mathcal{D}\setminus \{0\}$ then $(\nabla_{X^{h}}^{h}F)\xi^{h}\neq0$, where $\mathcal{D}$ is     defined by (\ref{df1}).
\end{proof}
We define the fundamental $2$-form $\Phi'$ by
\begin{equation*}
\Phi'(\tilde{X},\tilde{Y})=G(\tilde{X},F\tilde{Y})-\frac{p}{2}G(\tilde{X},\tilde{Y}),\,\,\,\,\,\tilde{X},\tilde{Y}\in \mathcal{X}(TM).
\end{equation*}
\begin{proposition}
Let $M$ be a P-Sasakian manifold with structure tensor $(\phi, \eta,\xi,g)$ and   $F$ be defined by (\ref{e2}). Then  fundamental $2$-form $\Phi'$,
is  never closed.
\end{proposition}
\begin{proof}
Let $X\in \mathcal{D}$ be  an unit vector field i.e. $g(X,X)=1$, where $\mathcal{D}$ is    defined by (\ref{df1}). Then  by  similar proof of Proposition  \ref{pp1}, we have
\begin{eqnarray*}
-\frac{6}{2\sigma_{p,q}-p}d\Phi'(X^{h},X^{v}, \xi^{v})&=&-G([\xi^{v},X^{h}],(\phi X)^{v})=-g(\nabla_{X}\xi ,\phi X)^{v}\\&=&-g(X,X)^{v}=-1.
\end{eqnarray*}
\end{proof}


\begin{thebibliography}{99}
\bibitem{G1}  S. I. Goldenberg, K. Yano, Polynomial structures on  manifolds, Kodai Math. Sem. Rep. 22 (1970) 199-218.
\bibitem{G2} S. I. Goldenberg, N. C. Petridis, Differentiable  solutions of algebraic  equations on manifolds, Kodai Math.  Sem. Rep. 25(1973)  111-128.
\bibitem{CEH}C. E. Hretcanu, M. Crasmareanu, Metallic  structures on Riemannian manifolds, Revista de la Unin mathematica Argentina, 54(2) (2013) 15-27.
\bibitem{MLP}M. de Le\'on, P. R. Rodrigues, Methods of differential geometry in analytical mechanics, North-Holland, 1989.
\bibitem{VOP}V. Oproiu, A generalization of natural almost Hermitian structures on the tangent
bundles, Math. J. Toyama Univ., 22 (1999), 1-14.
\bibitem{IS} I. Sato, On  a structure similar to the almost contact structure II,  Tensor N.S. 31(1977) 199-205.
\bibitem{YI} K. Yano, S. Ishihara, Tangent and cotangent bundle, Marcel Dekker Inc., New York, 1973.
\end{thebibliography}
\end{document}